\documentclass{article}
\usepackage{amsmath,amssymb,mathrsfs,amsthm}
\usepackage{array,enumerate}
\usepackage{latexsym}
\usepackage{mathrsfs}
\usepackage{amscd}
\usepackage[all]{xy}
\usepackage{graphicx}
\xyoption{pdf}
\usepackage{url}
\usepackage{comment}
\usepackage{framed, color}
\usepackage{ulem}

\definecolor{shadecolor}{gray}{0.80}

\theoremstyle{plain}
\newtheorem{thm}{Theorem}[section]
\newtheorem*{thm*}{Theorem}
\makeatletter
    \let\c@equation\c@thm
    
    \@addtoreset{equation}{section}
\makeatother

\newtheorem*{cor*}{Corollary}
\newtheorem{prop}[thm]{Proposition}
\newtheorem*{prop*}{Proposition}
\newtheorem{lem}[thm]{Lemma}
\newtheorem*{lem*}{Lemma}
\theoremstyle{definition}
\newtheorem{dfn}[thm]{Definition}
\newtheorem*{dfn*}{Definition}
\newtheorem{rem}[thm]{Remark}
\newtheorem*{rem*}{Remark}

\newtheorem*{ex*}{Example}

\newtheorem*{cnst*}{Construction}

\newtheorem*{pers*}{Perspective}

\newtheorem*{dig*}{Digression}
\newtheorem{claim}[thm]{Claim}
\newtheorem*{claim*}{Claim}

\usepackage{hyperref}

\title{Suslin's moving lemma with modulus}
\author{Wataru Kai
	\thanks{Graduate School of Mathematical Sciences, the University of Tokyo.
	JSPS Research Fellow (JSPS KAKENHI Grant Number 15J02264).
	Supported by the Program for Leading Graduate 
	Schools, MEXT, Japan.
	\newline
	\textit{E-mail address}: kaiw@ms.u-tokyo.ac.jp}, 
Hiroyasu Miyazaki
	\thanks{Graduate School of Mathematical Sciences, the University of Tokyo.
	JSPS Research Fellow (JSPS KAKENHI Grant Number 15J08833).
	Supported by the Program for Leading Graduate Schools, MEXT, Japan.
	\newline
	\textit{E-mail address}: miyazaki@ms.u-tokyo.ac.jp} }
	
\begin{document}

\maketitle

\begin{abstract}
The moving lemma of Suslin states that a cycle on $X\times \mathbb{A} ^n$ meeting all faces properly can be moved so that it becomes equidimensional over $\mathbb{A}^n$. This leads to an isomorphism of motivic Borel-Moore homology and higher Chow groups.

In this short paper we formulate and prove a variant of this. It leads to an isomorphism of Suslin homology with modulus and higher Chow groups with modulus, 
in an appropriate pro setting.

\end{abstract}

\section{Introduction}

Suslin \cite{Suslin} has proved roughly that a cycle on $X\times \mathbb{A} ^n$ meeting all faces properly can be moved so that it becomes equidimensional over $\mathbb{A}^n$. Here $X$ is an affine variety over a base field $k$.
As a consequence he obtains that the inclusion ($r\ge 0$)
\[ z^{\mathit{equi}}_r(X,\bullet )\hookrightarrow z_r(X,\bullet ) \]
of the cycle complex of equidimensional cycles into Bloch's cycle complex is a quasi-isomorphism.

Recently the context has been extended to cycles with modulus by Binda-Kerz-Saito \cite{KS, BS} and Kahn-Saito-Yamazaki \cite{KSY}. The reader finds the definitions below. 
There is an obvious injection ($r\ge 0$)
\begin{equation*}\label{inc} z^{\mathit{equi}}_r (\overline{X}|Y,\bullet ) \hookrightarrow z_r(\overline{X}|Y, \bullet ) \end{equation*}
for each pair $(\overline{X},Y)$ consisting of a finite type $k$-scheme $\overline{X}$ and an effective Cartier divisor $Y$ on it. We usually write $X:=\overline{X}\setminus Y$.

\

In this paper
we prove a variant of Suslin's moving lemma which takes the modulus condition into account (Theorem \ref{suslin-modulus} below). 
Our version of Suslin's moving lemma implies the following:
\begin{thm}[Theorem \ref{pro-comparison}]
Suppose $\overline{X}$ is affine and $X$ is an open set of $\overline{X}$ such that $\overline{X}\setminus X$ is the support of an 
effective Cartier divisor $Y$. Let $r\ge 0$.
Then the inclusions ($m\ge 0$)
\[ z^{\mathit{equi}}_r (\overline{X}|mY,\bullet )\subset z_r(\overline{X}|mY,\bullet ) \]
induce an isomorphism of abelian groups
\[ \varprojlim _{m }  
H_n(z^{\mathit{equi}}_r (\overline{X}|mY,\bullet ))\cong 
\varprojlim _{m}
\mathrm{CH}_r(\overline{X}|mY,n). \]
\end{thm}
Actually we can prove an isomorphism of pro-abelian groups. We do not know if 
the inclusion is a quasi-isomorphism before taking limits.

\

Now we recall the definitions.
We set $\square ^n:=(\mathbb{P}^1\setminus \{ \infty \} )^n=\mathrm{Spec}(k[y_1,\dots ,y_n])$ in this paper, contradicting some authors who prefer $1$ as the point at infinity. With this convention our computations look simpler.
We set a divisor on $(\mathbb{P}^1)^n$:
\[ F_n = \sum _{i=1}^n (\mathbb{P}^1)^{i-1}\times \{ \infty \} \times (\mathbb{P}^1)^{n-i} .\]
The {\it faces} of $\square ^n$ are $\{ y_i=0\} $, $\{ y_i=1 \}$ and their intersections.

\begin{dfn}[{\cite{BS}, \cite{KSY}}]
(1) Let $\underline{z}_r(\overline{X}|Y,n)$ be the group of $(r+n)$-dimensional cycles on $X\times \square ^n$ whose components $V$ meet faces of $\square ^n$ properly, and {\it have modulus $Y$}, i.e.:
\begin{quote}
Let $\overline{V}^N$ be the normalization of $\overline{V}\subset \overline{X}\times (\mathbb{P}^1)^n$, the closure of $V$. Let $\varphi _V\colon \overline{V}^N \to \overline{X}\times (\mathbb{P}^1)^n$ be the natural map. Then the inequality of Cartier divisors
\[ \varphi _V^{-1}(Y\times (\mathbb{P}^1)^n)\le \varphi _V^{-1}(\overline{X}\times F_n) \]
holds. (When $n=0$ the condition reads: the closure $\overline{V}\subset \overline{X}$ of $V$ is contained in $X$ i.e.~$V=\overline{V}$.)
\end{quote}
Let $\partial _{i,0}\colon \square ^{n-1}\hookrightarrow \square ^n$ be the embedding of the face $\{ y_i=0 \} $:
\[ \partial _{i,0}\colon (y_1,\dots ,y_{n-1})\mapsto (y_1,\dots ,\overset{i}{\check{0}},y_i,\dots ,y_{n-1}). \]
Define $\partial _{i,1}$ similarly.
The groups $\underline{z}_r(\overline{X}|Y,n)$ form a complex by the differentials
\[ \sum _{i=1}^n (-1)^{i}(\partial _{i,1}^*-\partial _{i,0}^*)\colon \underline{z}_r(\overline{X}|Y,n)\to \underline{z}_r(\overline{X}|Y,n-1).  \]

(2) Let $\underline{z}^{\mathit{equi}}_r(\overline{X}|Y,n)$ be the subgroup of $\underline{z}_r(\overline{X}|Y,n)$ consisting of cycles that are equidimensional over $\square ^n$ (necessarily of relative dimension $r$). 
\end{dfn}

\begin{rem}
The condition that $V$ has modulus $Y$ makes sense for any closed subset $V$ of ${X}\times \square ^n$. (In that case, normalization of a closed subset means the disjoint union of the normalizations of its reduced irreducible components.)
\end{rem}

\begin{dfn}
We define the {\it degenerate part} $\underline{z}_r(\overline{X}|Y,n)_{\mathrm{degn}} \subset \underline{z}_r(\overline{X}|Y,n)$ as the subgroup generated by the cycles of the form $(\mathrm{id}_{\overline{X}} \times \mathrm{pr}_i )^\ast (V)$, where $V \in \underline{z}_r(\overline{X}|Y,n-1)$ and $\mathrm{pr}_i : \square^n \to \square^{n-1}, (y_1 ,\dots ,y_n ) \mapsto (y_1 ,\dots ,y_{i-1},y_{i+1},\dots ,y_n )$ for some $i=1,\dots n$.
We also define the {\it degenerate part} $\underline{z}_r^{\mathit{equi}} (\overline{X}|Y,n)_{\mathrm{degn}} \subset \underline{z}_r^{\mathit{equi}} (\overline{X}|Y,n)$ in a similar way.
We set
\begin{align*}
z_r(\overline{X}|Y,n)&:=\underline{z}_r(\overline{X}|Y,n)/\underline{z}_r(\overline{X}|Y,n)_{\mathrm{degn}}, \\
z_r^{\mathit{equi}} (\overline{X}|Y,n)&:=\underline{z}_r^{\mathit{equi}} (\overline{X}|Y,n)/\underline{z}_r^{\mathit{equi}} (\overline{X}|Y,n)_{\mathrm{degn}} .
\end{align*}
Noting that the differentials $\partial_{i,0} ,\partial_{i,1}$ preserve degenerate parts, we can see that $z_r(\overline{X}|Y,n)$ and $z_r^{\mathit{equi}} (\overline{X}|Y,n)$ also form complexes.
We define the {\it higher Chow group with modulus} by
\begin{align*}
\mathrm{CH}_r (\overline{X}|Y,n) &:= H_n (z_r(\overline{X}|Y,\bullet )).
\end{align*}
\end{dfn}

\begin{rem}\label{degn}
The subgroups
\[ \underline{z}_r (\overline{X}|Y,n )_0 := \bigcap _{i=1}^n\ker (\partial ^*_{i,0})\subset \underline{z}_r(\overline{X}|Y,n) \]
form a subcomplex. One checks that the composite map 
\[
 \underline{z}_r (\overline{X}|Y,\bullet )_0 \hookrightarrow \underline{z}_r(\overline{X}|Y,\bullet ) \to z_r(\overline{X}|Y,\bullet )
\]
is an isomorphism of complexes.
Using this, we have a direct sum decomposition
\[ \underline{z}_r(\overline{X}|Y,\bullet )=z_r(\overline{X}|Y,\bullet )\oplus \underline{z}_r(\overline{X}|Y,\bullet )_{\mathrm{degn}}   \]
of a complex. We have a similar decomposition of $\underline{z}^{\mathit{equi}}_r(\overline{X}|Y,\bullet )$, and the inclusion $\underline{z}^{\mathit{equi}}_r(\overline{X}|Y,\bullet )\hookrightarrow \underline{z}_r(\overline{X}|Y,\bullet )$ is compatible with the decompositions.

\end{rem}

\section{Equidimensionality theorem}\label{sec-equidim}

Let $k$ be an infinite base field.
We will formulate and prove a variant of Suslin's equidimensionality theorem for modulus pairs $(\overline{X},Y)$ (i.e.~a $k$-scheme $\overline{X}$ equipped with an effective Cartier divisor $Y$) for which $\overline{X}$ is affine.

\subsection{Suslin's generic equidimensionality theorem (review)}

\begin{thm}[{\cite[Th.1.1]{Suslin}}]\label{suslin-original}
Assume that $\overline{X}$ is an affine scheme, $V$ is a closed subscheme in $\overline{X}\times \square ^n$ and $t$ is a nonnegative integer such that $\dim V \le n+t$. Assume further that $Z$ is an effective divisor in $\square ^n$ and $\varphi :\overline{X}\times Z\to \overline{X}\times \square ^n$ is any $\overline{X}$-morphism. Then there exists an $\overline{X}$-morphism $\Phi :\overline{X}\times \square ^n\to \overline{X}\times \square ^n$ such that
\begin{enumerate}[(1)]
\item $\Phi |_{\overline{X}\times Z}=\varphi $
\item Fibers of the projection $\Phi ^{-1}(V)\to \square ^n$ over points of $\square ^n\setminus Z$ have dimension $\le t$.
\end{enumerate}
\end{thm}

\begin{proof}[Sketch of proof]
Note that the $\overline{X}$-morphisms $\varphi ,\Phi $ are determined by $n$ regular functions on $\overline{X}\times Z$ and $\overline{X}\times \square ^n$ respectively.

We can reduce the problem to the case $\overline{X}=\mathbb{A}^m$ as follows. Take any closed embedding $\overline{X}\hookrightarrow \mathbb{A}^m$ and regard $V$ as a subset of $\mathbb{A}^m\times \square ^n$. By the above observation, the given $\varphi $ can be extended to an $\mathbb{A}^m$-morphism $\varphi \colon \mathbb{A}^m\times Z\to \mathbb{A}^m\times \square ^n$.
Suppose we have found an $\mathbb{A}^m$-morphism $\Phi \colon \mathbb{A}^m\times \square ^n\to \mathbb{A}^m\times \square ^n$ with the desired properties for $\mathbb{A}^m$ and $V$. It restricts to an $\overline{X}$-morphism $\overline{X}\times \square ^n\to \overline{X}\times \square ^n$ and satisfies the desired properties for $\overline{X}$ and $V$.

From now on we assume $\overline{X}=\mathbb{A}^m$. Let $x_1,\dots ,x_m$ be the coordinates of 
$\mathbb{A}^m$ and $y_1,\dots ,y_n$
be the coordinates of $\square ^n$.
Let $h(\underline{y})$ be the defining equation of $Z\subset \square ^n$.

We are given a $\overline{X}$-morphism 
$\varphi :\overline{X}\times Z\to \overline{X}\times \square ^n$
i.e.~a $k$-morphism $\overline{X}\times Z \to \square ^n$. It corresponds to $k$-algebra homomorphisms
\[ \begin{array}{ccccc} 
k[y_1,\dots ,y_n]&\to & k[x_1,\dots ,x_m,y_1,\dots ,y_n]/(h(\underline{y})) 
\\%
y_i &\mapsto & f_i(\underline{x},\underline{y})\mod (h(\underline{y})) .
\end{array}\]
Suslin constructs the desired morphism $\Phi$ as the morphism corresponding to homomorphisms of the form
\[ \begin{array}{ccccc} 
k[y_1,\dots ,y_n]&\to & k[x_1,\dots ,x_m,y_1,\dots ,y_n] 
\\%
y_i &\mapsto &\Phi _i:=f_i(\underline{x},\underline{y})+h(\underline{y})F_i(\underline{x})
\end{array}\]
where $F_1(\underline{x}),\dots ,F_n(\underline{x})$
are appropriately chosen homogeneous polynomials of degree $N$ for a large $N$.

He has shown that if we take $N$ large enough, then a generic choice of $F_1, \dots ,F_n$ makes the equidimensionality condition true.
\end{proof}

\subsection{Suslin's equidimensionality theorem, with modulus}

Recall a {\it face} of $\square ^n=\mathrm{Spec}(k[y_1,\dots ,y_n])$ is a closed subscheme of the form $\{ y_i=0 \} , \{ y_i =1\} $ or an intersection of them. Put $\partial \square ^n=\cup _{\partial }\partial (\square ^{n-1})$ where $\partial \colon \square ^{n-1}\hookrightarrow \square ^n$ runs through embeddings of codimension $1$ faces. It is a closed subset defined by the equation $h(\underline{y})=y_1(1-y_1)\dots y_n(1-y_n)$.

We need the following version of Suslin's moving lemma where we control the degrees of the map $\Phi $.

\begin{thm}\label{suslin-degree}
Let $\overline{X}=\mathrm{Spec}(R)$ be an affine $k$-scheme and $V\subset \overline{X}\times \square ^n$ be a closed subset of dimension $n+t$ for some $t\ge 0$. 
Suppose given a morphism 
\[ \Phi '\colon \overline{X}\times \partial \square ^{n}\to \overline{X}\times \square ^{n} \] 
and
suppose there is an integer $d\ge 2$ such that 
for any codimension $1$ face $\partial \colon \square ^{n-1}\hookrightarrow \square ^n$, the composite $\Phi '\circ (\mathrm{id}_{\overline{X}}\times \partial )$
is defined by polynomials $\Phi '_{i,\partial }\in R[y_1,\dots ,y_{n-1}]$ ($1\le i\le n$) whose degrees with respect to $y_j$ are at most $d$ for each $j$.

Then we can find an $\overline{X}$-map
\[ \Phi ^n :\overline{X}\times \square ^n \to \overline{X}\times \square ^n \]
extending $\Phi '$
such that $(\Phi ^n)^{-1}(V )\subset \overline{X}\times \square ^n$ has fibers of dimension $\le t$ over $\square ^n\setminus \partial \square ^{n}$,
and moreover, the functions $\Phi ^n_i\in R[y_1,\dots ,y_n]$ defining $\Phi ^n$ ($1\le i\le n$) have degrees $\le d$ with respect to each $y_j$.
\end{thm}

\begin{proof}
The map $\Phi '$ is determined by $R$-coefficient polynomials $f_i(y_1,\dots ,y_n) \mod h(\underline{y})$ ($1\le i\le n$). If we substitute $y_j=0 $ or $y_j=1$ to $f_i$ we get a polynomial which has degree $\le d$ with respect to each $y_k$ by the hypothesis.

\begin{lem}\label{modify}
Let $d\ge 1$ be an integer.
Suppose given a polynomial
$f(y_1,\dots ,y_n)\in R[y_1,\dots ,y_n]$ such that for each $j$ if we substitute any of $y_j=0$ or $y_j=1$, the resulting polynomial has degree $\le d$ with respect to each $y_k$. Then $f\mod h(\underline{y})$ has a (unique) representative which has degree $\le d$ with respect to each $y_j$ (where we keep the notation $h(\underline{y}):=y_1(1-y_1)\cdots y_n(1-y_n)$). 
\end{lem}
\begin{proof}
For each $i$ denote by $y_i(-|_{y_i=1})$ the operator which sends a polynomial $f$ to $y_i\cdot (f|_{y_i=1})$ and define $(1-y_i)(-|_{y_i=0})$ similarly. Note that for different $i$ and $j$ the operators $y_i(-|_{y_i=1})$ and $y_j(-|_{y_j=1})$ commute (and similarly for other pairs). Put $\alpha _i:=1- y_i(-|_{y_i=1}) - (1-y_i)(-|_{y_i=0})$.
Then one can see the polynomial
\[ f-\left( \alpha _1 \dots \alpha _n f\right) \]
is the desired representative.
\end{proof}

By the previous lemma, we can replace representatives $f_i(\underline{y})$ so that they have degrees $\le d$ with respect to each $y_j$.

By Suslin's proof of Theorem \ref{suslin-original}, there are elements $F_i\in R$ such that if we define
$\Phi ^n$ by setting its components as ($1\le i\le n$)
\[ \Phi ^n_i(\underline{y}):=f_i(\underline{y})+h(\underline{y})F_i , \]
then the condition on fiber dimensions is satisfied.
Moreover, from this form, $\Phi ^n_i$ has degree $\le d$ with respect to each $y_j$.
This completes the proof of Theorem \ref{suslin-degree}.
\end{proof}


\begin{lem*}[{Containment Lemma, \cite[Prop.2.4]{KP}}]
Let $V\subset \overline{X}\times \square ^n$ be a closed subset which has modulus $Y$ and $V'\subset V$ be a smaller closed subset. Then $V'$ also has modulus $Y$.
\end{lem*}

\begin{prop}\label{degree-modulus}
Let $(\overline{X},Y)$ be a modulus pair with $\overline{X}=\mathrm{Spec}(R)$ affine. Let $d$ be a positive integer and $V\subset \overline{X}\times \square ^n$ be a closed subset having modulus $nd\cdot Y$. Suppose $\Phi \colon \overline{X}\times \square ^{n'}\to \overline{X}\times \square ^n$ is an $\overline{X}$-morphism defined by polynomials $\Phi _j \in R[y_1,\dots ,y_{n'}]$ ($1\le j\le n$) having degrees $\le d$ with respect to each $y_i$. Then the closed subset $\Phi ^{-1}(V)$ of $\overline{X}\times \square ^{n'}$ has modulus $Y$.
\end{prop}
\begin{proof}
Since the assertion is local, we may assume $Y$ is principal and defined by $u\in R$.
Let $V'$ denote any one of the irreducible components of
$\Phi ^{-1}(V)$ and let $\overline{V'}^N$ be its normalization of its closure in $\overline{X}\times (\mathbb{P}^1)^{n'}$.

\[ \begin{array}{cccccccccccc}
\overline{V'}^{N}
\\%
\downarrow &&&&&&
\\%
\overline{V'}
&\supset 
&V'
& \subset 
&\Phi ^{-1}(V)
&\subset 
&\overline{X}\times \square ^{n'}
\\%
&&&&\downarrow &&\downarrow \Phi
\\%
&&&&V&\subset &\overline{X}\times \square ^n
\end{array} \]
Thanks to Containment Lemma above, the closure of $\Phi (V')\subset V$ has modulus $ndY$.
By replacing $V$ by the closure of $\Phi (V')$, we may assume the map $V'\to V$ is dominant.

\begin{claim}\label{domain-of-def}
Let $\overline{V'}^{N\circ }$ be the domain of definition of the rational map
\[ \overline{V'}^N\to \overline{X}\times (\mathbb{P}^1)^{n'}\overset{\Phi}{\dashrightarrow } \overline{X}\times (\mathbb{P}^1)^n.  \]
Then the complement of $\overline{V'}^{N\circ }$ in $\overline{V'}^N$ has codimension $\ge 2$.
\end{claim}
\begin{proof}
Let $v$ be a point of $\overline{V'}^N$ of codimension 1. Since the generic point $\eta$ of $\overline{V'}^N$ lands on $\overline{X}\times \square ^n$ we have a commutative diagram
\[ \xymatrix{
\eta \ar@{}|{\in}[r] \ar[rrd]
& \mathrm{Spec}\mathcal{O}_{v} \ar[r]
&\overline{X}\times (\mathbb{P}^1)^{n'} \ar[dr] \ar@{-->}^{\Phi}[d]
\\%
&
&\overline{X}\times (\mathbb{P}^1)^n \ar[r]
&\overline{X}
} \]
The assertion follows from the valuative criterion of properness (of the projection $\overline{X}\times (\mathbb{P}^1)^n\to \overline{X}$).
\end{proof}

By Claim \ref{domain-of-def}, we find that
a Cartier divisor on $\overline{V'}^N$ is effective if and only if its restriction to $\overline{V'}^{N\circ }$ is effective, since $\overline{V'}^{N }$ is normal.

\

Write $\mathrm{pr}_j:\overline{X}\times (\mathbb{P}^1)^n\to \mathbb{P}^1$ for the projection to the $j$-th $\mathbb{P}^1$ and $\Phi _j$ for the composite rational map
$\overline{X}\times (\mathbb{P}^1)^{n'} 
\overset{\Phi }{\dashrightarrow }
\overline{X}\times (\mathbb{P}^1)^n  \xrightarrow{\mathrm{pr}_j}\mathbb{P}^1$, also seen as a rational function on $\overline{X}\times (\mathbb{P}^1)^n$.
We will denote the pull-backs of $\Phi$ and $\Phi _j$ to $\overline{V'}^{N\circ }$ by $\Phi ^V$ and $\Phi ^V_j$. By definition of $\overline{V'}^{N\circ }$ they are well-defined morphisms from $\overline{V'}^{N\circ }$ to $\overline{X}\times (\mathbb{P}^1)^n$ and  to $\mathbb{P}^1$ respectively. There is a uniquely induced morphism $\overline{V'}^{N\circ }\to \overline{V}^N$ because now we are assuming $V'\to V$ is dominant.

For any given point of $\overline{V'}^{N\circ }$, we can find an affine open set $\mathrm{Spec}(A)\subset \overline{V}^N$ and an affine neighborhood $\mathrm{Spec}(B)\subset \overline{V'}^{N\circ }$ of the point such that $\Phi ^V$ restricts to a morphism $\Phi ^V \colon \mathrm{Spec}(B)\to \mathrm{Spec}(A)$. 
\[ \begin{array}{cccccccccc}
\mathrm{Spec}(B)
&\subset 
&\overline{V'}^{N\circ } 
&\to 
&\overline{X}\times (\mathbb{P}^1)^{n'}
\\%
&&\downarrow \Phi ^V&&\rotatebox{90}{$\dashleftarrow $}\Phi
\\%
\mathrm{Spec}(A)
&\subset 
&\overline{V}^N 
&\to 
&\overline{X}\times (\mathbb{P}^1)^n
\end{array} \]
By shrinking $\mathrm{Spec}(A)$ if necessary, we may assume $y_j$ or $1/y_j$ is regular on $\mathrm{Spec}(A)$ for each $j$. Denote by $J\subset \{ 1,\dots ,n\} $ the set of $j$'s for which $1/y_j$ is regular.
The divisor $F_n$ is defined by the equation $\frac{1}{\prod _{j\in J}y_j }=0$ on $\mathrm{Spec}(A)$. Since $V$ has modulus $ndY$, the rational function $\frac{1}{\prod _{j\in J}y_j}/u^{nd}$ on $\mathrm{Spec}(A)$ is regular. Pulling it back by $\Phi ^V$, we find that the rational function $\frac{1}{\prod _{j\in J}\Phi ^V_j}/u^{nd}$ on $\mathrm{Spec}(B)$ is regular.

Shrinking $\mathrm{Spec}(B)$ if necessary, we may assume $y_i$ or $1/y_i$ is regular on $\mathrm{Spec}(B)$ for each $i$.
Let $I\subset \{ 1,\dots ,n' \} $ be the set of $i$'s
for which $1/y_i$ is regular on $\mathrm{Spec}(B)$;
the divisor $F_{n'}$ is defined by $\frac{1}{\prod _{i\in I}y_i}=0$ on $\mathrm{Spec}(B)$.
\begin{claim}\label{regularity}
The rational function $\frac{\Phi ^V_j}{\prod _{i\in I}y_i^d}$ on $\mathrm{Spec}(B)$ is regular for each $j\in \{ 1,\dots ,n\} $ (i.e.~it is a morphism from $\mathrm{Spec}(B)$ into $\mathbb{A}^1\subset \mathbb{P}^1$).
\end{claim}
\begin{proof}
The function is the restriction of the meromorphic function $\frac{\Phi _j}{\prod _{i\in I}y_i^d}$ on $\overline{X}\times (\mathbb{P}^1)^{n'}$.
It is written as an $R$-coefficient polynomial in the variables $1/y_i$ ($i\in I$) and $y_i$ ($i\in I^c$) by the assumption on $\Phi $.
So it is regular around the (image of the) considered point on $\overline{X}\times (\mathbb{P}^1)^{n'}$. 
%
\end{proof}

By Claim \ref{regularity} the function
\[ \left( \frac{1}{\prod _{j\in J}\Phi _j}/u^{nd} \right)
\cdot \prod _{j\in J} \frac{\Phi _j}{\prod _{i\in I}y_i^d }=
\frac{1}{\prod _{i\in I}y_i^{d\cdot \# J} }/u^{nd} \]
is regular on $\mathrm{Spec}(B)$.
This shows a relation of Cartier divisors on $\mathrm{Spec}(B)$:
\[ nd\left( \prod _{i\in I}\frac{1}{y_i}\right) -nd(u) \ge 0  \]
which implies the relation
\[  (\text{pullback of }F_{n'}) -(\text{pullback of } Y) \ge 0 \]
on $\mathrm{Spec}(B)$, hence on $\overline{V'}^{N\circ }$,
which is valid on $\overline{V'}^N$ as well by the comment made after Claim \ref{domain-of-def}. This completes the proof of Proposition \ref{degree-modulus}.
\end{proof}

\begin{rem}
Under the hypotheses of Proposition \ref{degree-modulus}, we can prove that the morphism $\Phi $ is admissible \cite[Def.1.1]{KSY} for the pairs $((\mathbb{P}_R^1)^{n'},ndF_{n'}), ((\mathbb{P}_R^1)^{n},F_{n})$. It gives an alternative proof of Proposition \ref{degree-modulus}. We sketch the proof of the admissibility. We use the fact that admissibility can be checked after replacing the source by an open cover, or after blowing up $(\mathbb{P}^1)^{n'}$ by a closed subset outside $\square ^{n'}$.
Set $\eta _i=1/y_i$.
The scheme $(\mathbb{P}^1)^{n'}$ is covered by open subsets $U_I=\mathrm{Spec}(R[\eta _i,y_{i'} ~{}_{i\in I,i'\notin I}])$ where $I$ runs though the subsets of $\{ 1,\dots ,n'\} $.
On the region $U_I$, the rational function $\Phi _j$ is written as the ratio of the next two regular functions, by the assumption on $\Phi _j$.
\[ \Phi _j=\frac{\Phi ^{(I)}_j(\eta _i,y_{i'}) }{\prod _{i\in I}\eta _i^d}.  \]
We blow up $U_I$ by the ideal $(\Phi ^{(I)}_j,\prod _{i\in I}\eta _i^d)$. We perform this blow up for all $j\in \{ 1,\dots ,n\} $. The resulting scheme is covered by $2^n$ open subsets
\[ U_{IJ}=\mathrm{Spec}(R\left[ \eta _i,y_{i'} ~{}_{i\in I,i'\notin I}, \frac{\prod _{i\in I}\eta _i^d}{\Phi ^{(I)}_j(\eta _i,y_{i'}) } ,\frac{\Phi ^{(I)}_{j'}(\eta _i,y_{i'}) }{\prod _{i\in I}\eta _i^d} ~{}_{j\in J,j'\notin J}\right] )  \]
where $J$ runs through the subsets of $\{ 1,\dots ,n  \} $. The morphism $\Phi $ naturally extends to a morphism $\Phi \colon U_{IJ}\to U_J\subset (\mathbb{P}^1)^n$.

On $U_{IJ}$, the pull-back of $F_n$ by $\Phi $ is represented by the function $\prod _{j\in J}\frac{\prod _{i\in I}\eta _i^d}{\Phi ^{(I)}_j(\eta _i,y_{i'}) }$. The divisor $ndF_{n'}$ is represented by $\prod _{i\in I}\eta _i^{nd}$. Their ratio is
\[  \prod _i\eta _i^{(n-\# J)d}\cdot \prod _{j\in J}\Phi ^{(I)}_j  \]
which is a regular function on $U_{IJ}$. This proves the admissibility.
\end{rem}

From Theorem \ref{suslin-degree} and Proposition \ref{degree-modulus},
we get:
\begin{thm}\label{suslin-modulus}
Let $(\overline{X},Y)$ be a modulus pair with $\overline{X}$ affine, and $V\subset \overline{X}\times \square ^n$ be a purely $(n+t)$-dimensional closed subset for some $t\ge 0$. Suppose $V$ has modulus $2n\cdot Y$. Then there is a series of maps
\[ \Phi ^\bullet \colon \overline{X}\times \square ^\bullet \to \overline{X}\times \square ^\bullet   \]
compatible with face maps 
i.e.~for any codimension $1$ face $\partial \colon \square ^{m}\hookrightarrow \square ^{m+1}$, the following commutes:
\[ \xymatrix{
\overline{X}\times \square ^m \ar[r]^{\Phi ^m} \ar@{^{(}->}[d]^\partial
&\overline{X}\times \square ^m \ar@{^{(}->}[d]^\partial
\\%
\overline{X}\times \square ^{m+1} \ar[r]^{\Phi ^{m+1}}
&\overline{X}\times \square ^{m+1}
}  \]
such that the closed subset
\[ (\Phi ^n)^{-1}(V)\subset \overline{X}\times \square ^n  \]
is equidimensional over $\square ^n$ of relative dimension $t$, and has modulus $Y$.
In fact, the defining polynomials $\Phi ^m_i$ can be taken to have degree $\le 2$ for each variable $y_j$.
\end{thm}
It is proved by induction on $m$, starting with $\Phi ^0=\mathrm{id}$ which has degree $0$.

\section{Suslin homology with modulus and Higher Chow groups with modulus}

In this section, let $\overline{X}$ be an {\it affine} algebraic $k$-scheme and $X$ be an open subset such that $\overline{X}\setminus X$ is the support of an 
effective divisor.
The aim of this section is to prove the following theorem:

\begin{thm}\label{pro-comparison}
The inclusions
\[ z^{\mathit{equi}}_r (\overline{X}|Y,\bullet )\subset z_r(\overline{X}|Y,\bullet ) \]
induce pro-isomorphisms on the homology groups:
\[  ``\lim _{Y} "
H_n(z^{\mathit{equi}}_r (\overline{X}|Y,\bullet ))\cong 
``\lim _{Y} "
\mathrm{CH}_r(\overline{X}|Y,n) \]
where $Y$ runs through effective Cartier divisors with support $\overline{X}\setminus X$.
\end{thm}

\begin{rem}%
In the terminology of \cite[\S 6]{FI}, the above theorem can be expressed as: the map
$``\lim _Y"z^{\mathit{equi}}_r (\overline{X}|Y,\bullet )\to ``\lim _Y" z_r(\overline{X}|Y,\bullet )$
is a weak equivalence in the $\mathcal{H}_*$-model category of pro-complexes of abelian groups.
\end{rem}%

\begin{rem}\label{reduction}
In fact, we prove below that the inclusions
\[ \underline{z}^{\mathit{equi}}_r (\overline{X}|Y,\bullet )\subset \underline{z}_r(\overline{X}|Y,\bullet ) \]
induce pro-isomorphisms on the homology groups
\[  ``\lim _{Y} "
H_n(\underline{z}^{\mathit{equi}}_r (\overline{X}|Y,\bullet ))\cong 
``\lim _{Y} "
H_n(\underline{z}_r (\overline{X}|Y,\bullet )). \]
Then, by the canonical splitting we saw in Remark \ref{degn}, Theorem \ref{pro-comparison} is an immediate consequence of the last isomorphisms.

Theorem \ref{pro-comparison} is stated for a general base field. The proof can be easily reduced to the case over an infinite base field by a norm (trace) argument. In what follows, we will assume the base field $k$ is infinite so that we may use the results of \S \ref{sec-equidim}.
\end{rem}



\subsection{Construction of weak homotopy}

\begin{dfn}
Let $N$ be a positive integer.
Suppose that for any $0\leq n\leq N$, we are given a $\overline{X}$-morphism $\varphi_{n}:\overline{X}\times \square^{n}\to \overline{X}\times \square^{n}$ such that for any $0\leq j\leq n\leq N$ the following diagram is commutative:
\begin{eqnarray*}\label{cd0}
\xymatrix{
\overline{X}\times \square^{n-1} \ar[d]_{1_{\overline{X}}\times s_{j}} \ar[r]^{\varphi_{n-1}} & \overline{X}\times \square^{n-1} \ar[d]^{1_{\overline{X}}\times s_{j}}\\
\overline{X}\times \square^{n} \ar[r]^{\varphi_{n}} & \overline{X}\times \square^{n}.
}
\end{eqnarray*}
We define a subgroup ${}_{\varphi} \underline{z}_{r}(\overline{X}|Y,n) \subset \underline{z}_{r}(\overline{X}|Y,n)$ to be the free abelian group on the set of integral closed subschemes $V \subset X\times \square^{n}$ such that $[V]\in \underline{z}_{r}(\overline{X}|Y,n)$ and the pullback $\varphi_{n}^{\ast}[V]$ is defined and contained in $\underline{z}_{r}(\overline{X}|Y,n)$.
Then, ${}_{\varphi} \underline{z}_{r}(\overline{X}|Y,\bullet )$ defines a subcomplex of $\underline{z}_{r}(\overline{X}|Y,\bullet )$.
\end{dfn}



In the following, we fix a closed subscheme $V \subset \overline{X} \times \square^n$ whose irreducible components have modulus $2n \cdot Y$ and take $\varphi_n := \Phi^n$, where $\Phi^\bullet $ is the system of morphisms given in Theorem \ref{suslin-modulus}.

\begin{dfn}
Define for each $n\geq 0$ an abelian subgroup ${}_\Phi \underline{z}_r^- (\overline{X}|Y,n ) \subset  {}_\Phi z_r (\overline{X}|Y,n )$ by
\[ {}_\Phi \underline{z}_r^- (\overline{X}|Y,n ):= {}_\Phi \underline{z}_r (\overline{X}|Y,n ) \cap \underline{z}_r (\overline{X}|2nY,n ) \subset  {}_\Phi \underline{z}_r^- (\overline{X}|Y,n ). \]
Then, we get a subcomplex ${}_\Phi \underline{z}_r^- (\overline{X}|Y,\bullet ) \subset {}_\Phi \underline{z}_r (\overline{X}|Y,\bullet )$. 
\end{dfn}

\begin{lem}\label{weakhomotopy}
The homomorphisms
\[ {}_\Phi \underline{z}_r^- (\overline{X}|Y,\bullet )
\overset{(\Phi^\bullet )^*}{\underset{incl.}{\rightrightarrows}}
 \underline{z}_r(\overline{X}|Y,\bullet ) \]
 are weakly homotpic (i.e. their restriction to any finitely generated subcomplex are homotopic).
\end{lem}

\begin{proof}
To construct a weak homotopy as in the assertion, we fix a finite set of integral closed subschemes 
$\{V^{n}_{k}\}(\subset X\times \square^{n})\in {}_\Phi \underline{z}_r^-(\overline{X}|Y,n)$
which is closed under pullback along faces.
Denote by $C_{n}$ the free abelian group generated by $[V^{n}_{k}]$'s.
Then, we get a subcomplex $C_{\ast } \subset {}_\Phi \underline{z}_r^- (\overline{X}| Y, \bullet)$.
Since the subcomplexes of this form are cofinal in all finitely generated subcomplexes,
it suffices to prove that $(\varphi - \mathrm{incl.}|_{C_{\bullet}})$ is homotopic to zero.
For the proof, we construct a family of $\overline{X}$-morphisms $\tilde{\Phi}^{n}: \overline{X}\times \square^{n}\times \mathbb{A}^{1}\to \overline{X}\times \square^{n}\times \mathbb{A}^{1}$ which satisfies the following conditions:
\ 

(1) The following diagrams commute:
\begin{eqnarray*}\label{cd1}
\xymatrix{
\overline{X}\times \square^{n} \ar[d]_{i_{0}} \ar[r]^{\mathrm{id}} & \overline{X}\times \square^{n} \ar[d]^{i_{0}}\\
\overline{X}\times \square^{n} \ar[r]^{\tilde{\Phi}^{n}}\times \mathbb{A}^{1} & \overline{X}\times \square^{n}\times \mathbb{A}^{1},
}
\end{eqnarray*}
\begin{eqnarray*}\label{cd2}
\xymatrix{
\overline{X}\times \square^{n} \ar[d]_{i_{1}} \ar[r]^{\Phi^{n}} & \overline{X}\times \square^{n} \ar[d]^{i_{1}}\\
\overline{X}\times \square^{n} \ar[r]^{\tilde{\Phi}^{n}}\times \mathbb{A}^{1} & \overline{X}\times \square^{n}\times \mathbb{A}^{1},
}
\end{eqnarray*}
\begin{eqnarray*}\label{cd3}
\xymatrix{
\overline{X}\times \square^{n-1} \times \mathbb{A}^{1} \ar[d]_{1_{\overline{X}}\times s_{j}\times 1_{\mathbb{A}^{1}}} \ar[r]^{\tilde{\Phi}^{n-1}} & \overline{X}\times \square^{n-1} \times \mathbb{A}^{1}\ar[d]^{1_{\overline{X}}\times s_{j}\times 1_{\mathbb{A}^{1}}}\\
\overline{X}\times \square^{n} \ar[r]^{\tilde{\Phi}^{n}}\times \mathbb{A}^{1} & \overline{X}\times \square^{n}\times \mathbb{A}^{1}.
}
\end{eqnarray*}

(2)
Set
$Z:=(\square^{n}\times 0)+(\square^{n}\times 1)+\partial \square^{n} \times \mathbb{A}^{1} \subset \square^{n} \times \mathbb{A}^{1}$.
Then, for any point $z \in \square^{n} \times \mathbb{A}^{1}$ outside $Z$,
the dimension of the fiber over $z$ of the map 
$(\tilde{\Phi}^{n} )^{-1}(\cup_{k}V^{n}_{k}\times \mathbb{A}^{1})\subset \overline{X}\times \square^{n}\times \mathbb{A}^{1} \to \square^{n}\times \mathbb{A}^{1}$
is $\leq r$.

(3)
Every component of 
$(\tilde{\Phi}^{n} )^{-1}(V^{n}_{k} \times \mathbb{A}^{1})$
has modulus $Y$.

\ 

%
%
Given $\tilde{\Phi}^{\bullet}$ as above, we may define a homotopy $\sigma$ as 
$\sigma (V^n_k) :=(\tilde{\Phi}^{n} )^{\ast}(V^{n}_{k} \times \mathbb{A}^{1})$.

\

Now we construct $\tilde{\Phi}^{\bullet}$.
Actually each component of $\tilde{\Phi }^n$ will have degrees $\le 2$ in each variable $y_j$, which implies the condition (3) by Proposition \ref{degree-modulus} applied to $n'=n+1$.
Suppose we have constructed $\tilde{\Phi }^{n-1}$.
Via the isomorphism $\square ^n\times \mathbb{A}^1\cong \square ^{n+1}$, we have $Z\cong \partial \square ^{n+1}$.
Condition (1) for $\tilde{\Phi }^{n-1}$ implies that there is a glued $\overline{X}$-map
\[ \overline{X}\times Z \to \overline{X}\times \square ^n\times \mathbb{A}^1, \]
whose restrictions to codimension $1$ faces of $\square ^n\times \mathbb{A}^1\cong \square ^{n+1}$ are either $\mathrm{id}$, $\Phi^n$ or $\tilde{\Phi }^{n-1}$. By the induction hypothesis and Theorem \ref{suslin-modulus}, these are defined by polynomials whose degrees in $y_j$ are $\le 2$ for each $j$. 
Then by Theorem \ref{suslin-degree}, we obtain $\tilde{\Phi }^n$ having degrees $\le 2$ and satisfying (1)(2).
\end{proof}

\subsection{Proof of the comparison theorem}\label{proof-of-comparison}

Finally we can prove Theorem \ref{pro-comparison}.
In the following, we use the following abbreviations:
\[
C_{\bullet}^{Y}:=\underline{z}^{\mathit{equi}}_r (\overline{X}|Y,\bullet ), \ \ 
D_{\bullet}^{Y}:=\underline{z}_r(\overline{X}|Y,\bullet ),
\]
Let $f^Y : C_{\bullet}^Y \to D_{\bullet}^Y$ denote the natural inclusion

By Remark \ref{reduction}, it suffices to prove that 
$``\lim_{Y}" H_{n}C^{Y}_{\bullet} \stackrel{``\lim"H_n f^{Y}}{\longrightarrow} ``\lim_{Y}" H_{n}D^{Y}_{\bullet}$
is an isomorphism in the category of pro-abelian groups $\mathrm{pro}\mathchar`- \mathbf{Ab}$.
Since the functor $``\lim"$ is exact, 
the kernel and the cockerel of the map $``\lim"H_n f^{Y}$ is given by
$``\lim"\mathrm{Ker}(H_n f^{Y}), ``\lim"\mathrm{Coker}(H_n f^{Y})$.
We prove that these objects in 
$\mathrm{pro}\mathchar`- \mathbf{Ab}$ are the zero object.
Now we have the following elementary lemma:
\begin{lem}
An object $A=\{A^{\gamma}\}_{\gamma \in \Gamma} \in \mathrm{pro}\mathchar`- \mathbf{Ab}$ is the zero object if and only if for any $\gamma \in \Gamma$ there exists $\gamma' > \gamma$ such that the projection map $p^{\gamma'}_{\gamma}:A^{\gamma'}\to A^{\gamma}$ is the zero map.
\end{lem}
Therefore, we are reduced to showing the following
\begin{lem}\label{zeromaps}
For any principal effective divisor $Y$ and $n\ge 0$, there exists $N>1$ such that the projections $\mathrm{Ker}(H_n f^{NY}) \to \mathrm{Ker}(H_n f^{Y})$ and $\mathrm{Coker}(H_n f^{NY}) \to \mathrm{Coker}(H_n f^{Y})$ are the zero maps.
\end{lem}

\begin{proof}

We firstly prove that $\mathrm{Coker}(H_n f^{2nY}) \to \mathrm{Coker}(H_n f^{Y})$ is the zero map for any $n \geq 0$.
Take arbitrary element $W \in H_{n}D_{\bullet}^{2nY}$.
\marginpar{\tiny \color{blue} }
Applying Lemma \ref{weakhomotopy} for $\Phi$ given in Theorem \ref{suslin-modulus} with respect to $W$,  there exists $V \in C_{n}^{Y}$ such that $\mathrm{pr}^{2nY}_{Y} W = f^Y V$ holds in $H_{n}D^{Y}_{\bullet}$.
This means that $(H_{n}D_{\bullet}^{2nY} \twoheadrightarrow ) \mathrm{Coker}(H_n f^{2nY}) \to \mathrm{Coker}(H_n f^{Y})$ is the zero map.

\ 

Next we prove that $\mathrm{Ker}(H_n f^{(2n+2)Y}) \to \mathrm{Ker}(H_n f^{Y})$ is the zero map.
Take $V \in \mathrm{Ker}(H_n f^{(2n+2)Y}) \subset H_n C_{\bullet}^{(2n+2)Y}$ arbitrarily.
In the following, we regard $V$ as an element of $C_{n}^{(2n+2)Y} \stackrel{f^{(2n+2)Y}}{\hookrightarrow} D_{n}^{(2n+2)Y}$.
Then, there exists $W \in D_{n+1}^{(2n+2)Y}$ such that $V = dW$ holds in $D_{n}^{(2n+2)Y}$.
It suffices to show that 
$\mathrm{pr}^{(2n+2)Y}_Y dW$ belongs to $dC_{n+1}^{Y}$.

Let $\Phi$ be the morphism as in Theorem \ref{suslin-modulus} corresponding to $W$.
Since we have $(\Phi^{n+1} )^\ast W \in C_{n+1}^{Y}$ by Theorem \ref{suslin-modulus},
it is equivalent to verify $d(\Phi^{n+1} )^\ast W - dW \in dC_{n+1}^{Y}$.
This element can be rewritten as 
\begin{align*}
&d((\Phi^{n+1} )^\ast - \mathrm{incl.}) W \\
&=d(d\sigma_{n+1} - \sigma_n d) W \\
&= -d\sigma_n dW \\
&= -d(\tilde{\Phi}^n )^\ast (d(W) \times \mathbb{A}^1),
\end{align*}
where  $\tilde{\Phi}$ and $\sigma$ are defined in the proof of Lemma \ref{weakhomotopy}.
By construction of $\tilde{\Phi}$, we can see that $(\tilde{\Phi}^n )^\ast (d(W) \times \mathbb{A}^1)$ is equidimesional.
Therefore, the right hand side of the equations belongs to $dC_{n+1}^Y$, which proves the desired assertion.
\end{proof}

\subsection{A consequence on the relative motivic cohomologies}
We can naturally sheafify our objects and consider the inclusion
\[ z^{\mathit{equi}}_r (\overline{X}|Y,\bullet )_{\mathrm{Zar}}\subset z_r(\overline{X}|Y,\bullet )_{\mathrm{Zar}}  \]
of Zariski sheaves of complexes on $\overline{X}$. The induced maps on homology sheaves
\[ ``\lim _Y"H_n(z^{\mathit{equi}}_r (\overline{X}|Y,\bullet )_{\mathrm{Zar}})
\xrightarrow{``\lim " f_n^Y}
``\lim _Y" \mathrm{CH}_r(\overline{X}|Y,n )_{\mathrm{Zar}} \]
are pro-isomorphisms of Zariski sheaves for all $n$. Indeed, by Lemma \ref{zeromaps}, the maps of sheaves
\[ \mathrm{Coker}(f_n^{2nY})\to \mathrm{Coker}(f_n^{Y})   \]
\[ \mathrm{Ker}(f_n^{(2n+2)Y})\to \mathrm{Ker}(f_n^{Y})   \]
are zero.

As a general fact on pro-categories, the functors $H^n_{\mathrm{Zar}}(\overline{X},-)$ extend to functors
\[ \begin{array}{ccc}
\text{pro-sheaves}&\to &\text{pro-abelian groups}\\
``\lim _i" F_i &\mapsto &``\lim _i"H^n_{\mathrm{Zar}}(\overline{X},F_i).
\end{array}  \]
We have hypercohomology spectral sequences in the abelian category of pro-abelian groups:
\begin{gather*}
E_2^{pq}=``\lim _Y"H^p_{\mathrm{Zar}}(\overline{X},H_{-q}(z^{\mathit{equi}}_r (\overline{X}|Y,\bullet )_{\mathrm{Zar}})) \Rightarrow ``\lim _Y"\mathbf{H}^{p+q}_{\mathrm{Zar}}(\overline{X},z^{\mathit{equi}}_r (\overline{X}|Y,\bullet )_{\mathrm{Zar}})
\\%
'E_2^{pq}=``\lim _Y"H^p_{\mathrm{Zar}}(\overline{X},CH_r(\overline{X}|Y,-q)_{\mathrm{Zar}})) \Rightarrow ``\lim _Y"\mathbf{H}^{p+q}_{\mathrm{Zar}}(\overline{X},z_r(\overline{X}|Y,\bullet )_{\mathrm{Zar}})
\end{gather*}
which are bounded to the range
$0\le p\le \dim \overline{X}$ and $q\le 0$.
Since the natural map
$E\to {}'E$ of spectral sequences induces isomorphisms on $E_2$-terms, we get isomorphisms
\[``\lim _Y" \mathbf{H}^{n}_{\mathrm{Zar}}(\overline{X},z^{\mathit{equi}}_r (\overline{X}|Y,\bullet )_{\mathrm{Zar}})\to 
``\lim _Y" \mathbf{H}^{n}_{\mathrm{Zar}}(\overline{X},z_r(\overline{X}|Y,\bullet )_{\mathrm{Zar}}).\]
So we have proved:
\begin{thm}%
For any algebraic scheme $\overline{X}$ and an effective Cartier divisor $Y_0$ on $\overline{X}$,
the natural maps of pro-abelian groups
\[``\lim _{Y}\text{'' }\mathbf{H}^{n}_{\mathrm{Zar}}(\overline{X},z^{\mathit{equi}}_r (\overline{X}|Y,\bullet )_{\mathrm{Zar}})\to 
``\lim _{Y}\text{'' } \mathbf{H}^{n}_{\mathrm{Zar}}(\overline{X},z_r(\overline{X}|Y,\bullet )_{\mathrm{Zar}}) \]
are isomorphisms, where $Y$ runs through effective Cartier divisors with support $|Y_0|$.
\end{thm}%

\end{document}